\documentclass[a4paper,12pt]{amsart}
\usepackage{amsmath, amssymb, amsthm}
\usepackage{verbatim}
\usepackage{hyperref}
\usepackage{xypic}

\newcounter{theorem}
\newtheorem{thm}[theorem]{Theorem}
\newtheorem{lemma}[theorem]{Lemma}
\newtheorem{prop}[theorem]{Proposition}
\newtheorem{cor}[theorem]{Corollary}

\theoremstyle{remark}
\newtheorem*{remark*}{Remark}
\newtheorem{remark}[theorem]{Remark}
\newtheorem{question}[theorem]{Question}

\numberwithin{equation}{section}
\numberwithin{theorem}{section}

\newcommand{\N}{\mathbb{N}}
\renewcommand{\setminus}{\backslash}

\newcommand{\id}{\mathrm{id}}

\newcommand{\Tor}{\mathrm{Tor}_1^{\Z}}
\newcommand{\Ext}{\mathrm{Ext}_1^{\Z}}
\newcommand{\Q}{\mathbb Q}
\newcommand{\Z}{\mathbb Z}

\begin{document}

\title[Characterization of approximately inner flip]{K-theoretic characterization of C*-algebras with approximately inner flip}
\author{Aaron Tikuisis}
\address{Institute of Mathematics\\
University of Aberdeen\\
Aberdeen, UK AB24 3UE}
\urladdr{http://homepages.abdn.ac.uk/a.tikuisis/}
\email{a.tikuisis@abdn.ac.uk}

\begin{abstract}
It is determined exactly which classifiable C*-algebras have approximately inner flip.
The answer includes a number of C*-algebras with torsion in their K-theory, and a number of C*-algebras that are self-absorbing but not strongly self-absorbing.
\end{abstract}

\thanks{The author is supported by an NSERC PDF}


\maketitle

\section{Introduction}

The concept of approximately inner flip for C*-algebras was first studied by Effros and Rosenberg in \cite{EffrosRosenberg}.
The concept has a close connection to strongly self-absorbing algebras defined by Toms and Winter \cite{TomsWinter:ssa}, a prominent idea in the Elliott classification programme \cite{BBSTWW,DadarlatPennig,DadarlatWinter:Trivialization,ENST,ElliottToms,GongLinNiu,HirshbergRordamWinter,MatuiSato:comp,MatuiSato:dr,Rordam:Z,SWW:Znucdim,TW:Zdr,TomsWinter:ZASH,Winter:drZstable,Winter:pure,Winter:localizing}.

Effros and Rosenberg showed that the class of C*-algebras with approximately inner flip is fairly restricted: it is contained in the class of simple, nuclear C*-algebras with at most one trace.
By considering the case of AF algebras, they showed that, in fact, approximately inner flip entails considerably more restrictions than just simplicity, nuclearity, and at most one trace.
Their result \cite[Theorem 3.9]{EffrosRosenberg} is that an AF algebra with approximately inner flip is stably isomorphic to a UHF algebra.
This is easily reformulated as a K-theoretic characterization of AF algebras with approximately inner flip: an AF algebra $A$ has approximately inner flip if and only if $K_0(A)$ is a subset of $\Q$.

This article generalizes this K-theoretic analysis of approximately inner flip, subject to the Universal Coefficient Theorem (UCT) (without which, serious K-theoretic computations are hopeless).
It is shown that if $A$ has approximately inner flip and satisfies the UCT, then $K_*(A):=K_0(A)\oplus K_1(A)$ is isomorphic (as an ungraded group) to one of the following groups:
\begin{enumerate}
\item $0$;
\item $\Z$;
\item $\Q_n$;
\item $\Q_m/\Z$;
\item $\Q_n \oplus \Q_m/\Z$,
\end{enumerate}
where in (iii)-(v), $n$ and $m$ are supernatural numbers of infinite type such that $m$ divides $n$ (see Section \ref{sec:Notation} for the definition of $\Q_n$).
The result is tight, in that every one of these groups (with any $\Z_2$-grading and any unperforated ordering) does arise as $K_*(A)$ for some C*-algebra with approximately inner flip.
In fact, if $A$ is a classifiable C*-algebra (in the sense of \cite{Phillips:piClass} and \cite{GongLinNiu}), then $A$ has approximately inner flip if and only if $K_*(A)$ is one of these groups and $A$ has at most one trace.
Interestingly, this provides a significant number of C*-algebras, even self-absorbing ones, with torsion in their K-theory (namely, from cases (iv) and (v)).

Some notable consequences of our results and methods are as follows, where $A$ is a classifiable C*-algebra:
\begin{enumerate}
\item If $A \otimes A$ has approximately inner flip then so does $A$;
\item If $A$ has approximately inner flip then it has asymptotically inner flip;
\item If $A$ has approximately inner flip then $A \otimes A \otimes \mathcal K$ is self-absorbing.
\end{enumerate}
(The first two of these are packaged into the main theorem, Theorem \ref{thm:MainThm}, while the third is Corollary \ref{cor:AAsa}.)
This prompts the question: which of these three facts can be shown without assuming that $A$ is classifiable?

Let us compare the situation to that of strongly self-absorbing C*-algebras.
Strongly self-absorbing C*-algebras have approximately inner flip, but have significantly more structure, allowing results such as (ii) to be proven without using classification, as done by Dadarlat and Winter \cite{DadarlatWinter:KKssa} (in fact, they show that any automorphism of a strongly self-absorbing C*-algebra is asymptotically inner).
Toms and Winter, in the same article that introduced the concept of strongly self-absorbing C*-algebras, established which possible K-theories can arise for strongly self-absorbing C*-algebra in the UCT class \cite[Proposition 5.1]{TomsWinter:ssa}; the main result of this article is a natural extension of their work.

Naturally, our main tool for analyzing the flip in terms of K-theory is the K\"unneth formula for C*-algebras due to Schochet \cite{Schochet:KunnethThm}, a short exact sequence relating $K_*(A\otimes A)$ to $K_*(A)\otimes K_*(A)$ and $\Tor(K_*(A),K_*(A))$.
However, one needs to know how the flip map $A \otimes A \to A \otimes A$ interacts with this exact sequence.
In Section \ref{sec:KunnethFlip}, we solve this problem; the map between the Tor components comes from the natural isomorphism $\Tor(G,H) \cong \Tor(H,G)$, which we first describe in Section \ref{sec:TorFlipMap}.

In Section \ref{sec:Family}, we introduce a family of C*-algebras (representatives for the groups listed above with possible gradings); having introduced notation for these C*-algebras, we state the main result, Theorem \ref{thm:MainThm}.
In Section \ref{sec:Suff}, we show the C*-algebras in this family each have approximately inner flip.
This result is useful in establishing, in Section \ref{sec:Necessary}, that if $A$ satisfies the UCT and has approximately inner flip, then $K_*(A)$ is one of the groups above.
Finally, in Section \ref{sec:Semigroup}, we explore the classifiable C*-algebras with approximately inner flip, by describing the semigroup of isomorphism classes of such C*-algebras, under the operation of $\otimes$.

\subsection*{Acknowledgments}

I would like to thank Bradd Hart, Ilijas Farah, and Mikael R\o rdam for discussions that contributed to this article.
In particular, Mikael R\o rdam suggested the problem, and Ilijas Farah made suggestions that led to a simpler proof of Proposition \ref{prop:PruferFlip}. 
I would moreover like to thank Sean Tilson for helping to explain a spectral sequence proof of Lemma \ref{lem:KunnethFlip} for topological K-theory.

\subsection{Notation}
\label{sec:Notation}

Let $A,B$ be C*-algebras.
Write $K_*(A):=K_0(A) \oplus K_1(A)$.
The suspension of $A$ is $SA:=C_0((0,1),A)$.
Write $A \otimes B$ to denote the minimal tensor product of $A$ and $B$.
The C*-algebras $A$ and $B$ are said to be stably isomorphic if they satisfy $A \otimes \mathcal K \cong B \otimes \mathcal K$ (in case $A$ and $B$ are separable, this is the same as Morita equivalence).
Denote by $\sigma_{A,B}$ the flip isomorphism $A \otimes B \to B \otimes A$, defined on elementary tensors by $\sigma_{A,B}(a \otimes b) = b \otimes a$.
Following \cite{EffrosRosenberg}, we say that $A$ has \textbf{approximately inner flip} if there is a net $(u_\lambda)$ of unitaries on $A$ such that
\[ \lim_\lambda \|u_\lambda xu_\lambda^* - \sigma_{A,A}(x)\| = 0, \quad x \in A \otimes A. \]
When $A$ is separable and has approximately inner flip, the net $(u_\lambda)$ can be taken to be a sequence.
If $A$ is separable, we say that $A$ has \textbf{asymptotically inner flip} if there is a continuous function $t \mapsto u_t$ from $[0,\infty)$ to the unitaries in $A \otimes A$, such that 
\[ \lim_{t\to\infty} \|u_t xu_t^* - \sigma_{A,A}(x)\| = 0, \quad x \in A \otimes A. \]

A \textbf{Kirchberg algebra} is a simple, separable, purely infinite, nuclear C*-algebra.
The class of Kirchberg algebras which satisfy the UCT has been classified by K-theory, by Kirchberg and Phillips (\cite{Kirchberg:piClass,Phillips:piClass}; see \cite[Chapter 8]{Rordam:ClassBook}).
We use $\mathcal O^\infty$ to denote the unital Kirchberg algebra in the UCT class that satisfies
\[ K_0(\mathcal O^\infty) \cong 0,\quad K_1(\mathcal O^\infty) \cong \Z. \]

Let $G,H$ be abelian groups.
For a prime number $p$, the group $G$ is a \textbf{$p$-group} if every $x \in G$ satisfies, for some $k\in\N$, $p^kx = 0$.
Denote by $\sigma_{G,H}$ the flip isomorphism $G \otimes H \to G \otimes H$, defined on elementary tensors by $\sigma_{G,H}(g \otimes h) = g \otimes h$.
(This should not be confused with the flip $\sigma_{A,B}$ on the C*-algebraic tensor product $A \otimes B$, because it is absurd to treat a C*-algebra as merely an abelian group.)

A \textbf{supernatural number} is a formal product
\begin{equation}
\label{eq:SNproduct}
 \prod_p p^{k_p},
\end{equation}
where the product is taken over all primes and $k_p \in \{0,1,2,\dots,\infty\}$ for each $p$.
Every natural number is a supernatural number.
Supernatural numbers may be multiplied (even infinitely many times) and the multiplication operation is also used to define what is meant by one supernatural number, $m$, dividing another, $n$ (in symbols, $m|n$).
A supernatural number $n$ is of infinite type if it is equal to its square (i.e., expressing $n$ as in \eqref{eq:SNproduct}, if $k_p \in \{0,\infty\}$ for each $p$); note that by this definition, $1$ is of infinite type.

If $n$ is a supernatural number, we define
\[ \Q_n := \Big\{\frac pq \in \Q \Big| p \in \mathbb Z, q \in \mathbb N, q|n\Big\}. \]
Note when $n=k^\infty$, $k \in \N$, then $\Q_n = \Z[1/k]$.
If $n$ is a supernatural number of infinite type, then $\Q_n/\Z \cong \bigoplus_p \Q_{p^\infty}/\Z$ where the direct sum is taken over all primes $p$ that divide $n$.

If $n$ is a supernatural number and $G$ is an abelian group, we say that a group $H$ is $n$-divisible if it is $p$-divisible for every prime $p$ which divides $n$.

\begin{lemma}
\label{lem:DivSplit}
Let $G$ be an abelian group and let $H \subset G$ be a subgroup, such that $G/H \cong \Q_{n}/\Z$ or $\Q_n$ for some supernatural number $n$.
If $H$ is $n$-divisible then $G \cong H \oplus G/H$.
\end{lemma}

\begin{proof}
Assume that $H$ is $n$-divisible, and let us show that the exact sequence 
\[
\xymatrix{
0 \ar[r] & H \ar[r] & G \ar[r]^-{\pi} & G/H \ar[r] & 0
} \]
splits.
Identify $G/H$ with $\Q_n/\Z$ or $\Q_n$ for notational simplicity.

Let $n=m_1m_2\cdots$ where $m_i \in \N$ for each $i$, and set $n_i:=m_1\cdots m_i$ for each $i \geq 0$.
Then
\[ \Q_n \cong \langle a_0,a_1,a_2,\dots \mid m_ia_i = a_{i-1}\rangle, \]
and
\[ \Q_n/\Z \cong \langle a_0,a_1,a_2,\dots \mid m_ia_i = a_{i-1}, a_0=0\rangle, \]
in both cases by identifying $\frac1{n_i}$ with $a_i$.

Therefore, to define a splitting of $\pi$, we need to find $b_i \in G$ such that $\pi(b_i) = \frac1{n_i}$ and $m_ib_i = b_{i-1}$ for all $i \geq 1$ and additionally, in case $G/H\cong \Q_n/\Z$, $b_0=0$.
We do this recursively.
If $G/H \cong \Q_n$, set $b_0$ equal to any lift of $1$; if $G/H \cong \Q_n/\Z$, set $b_0=0$.

Having defined $b_{i-1}$, choose any $c_i \in G$ such that $\pi(c_i)=\frac1{n_i}$.
Thus $\pi(m_ic_i - b_{i-1}) = 0$ so that $m_ic_i-b_{i-1} \in H$.
Since $H$ is $m_i$-divisible, there exists $z \in H$ such that $m_ic_i-b_{i-1}=m_iz$.
Thus, we may set $b_i:=c_i-z$, so that $\pi(b_i)=\pi(c_i)=\frac1{n_i}$ and $m_ib_i = b_{i-1}$ as required.
\end{proof}

\section{A family of C*-algebras}
\label{sec:Family}

For supernatural numbers $n$, $m_0$, and $m_1$, let $\mathcal E_{n,m_0,m_1}$ be the simple, separable, unital, $\mathcal Z$-stable, quasidiagonal C*-algebra with unique trace that satisfies
\begin{align*}
K_0(\mathcal E_{n,m_0,m_1})&=\Q_n \oplus \Q_{m_0}/\Z, \\
[1]_0&= 1 \oplus 0, \quad \text{and} \\
K_1(\mathcal E_{n,m_0,m_1})&=\Q_{m_1}/\Z.
\end{align*}

For supernatural numbers $m_0,m_1$, let $\mathcal F_{m_0,m_1}$ be the unital Kirchberg algebra in the UCT class that satisfies
\begin{align*}
K_0(\mathcal F_{m_0,m_1})&=\Q_{m_0}/\Z, \\
[1]_0&= 0, \quad \text{and} \\
K_1(\mathcal F_{m_0,m_1})&=\Q_{m_1}/\Z.
\end{align*}

\begin{remark}
\label{rmk:EFdef}
It is shown by Elliott in \cite[Theorem 5.2.3.2]{Elliott:ashrange} that the algebra $\mathcal E_{n,m_0,m_1}$ algebra exists, and Matui and Sato have shown that it is unique \cite[Corollary 6.2]{MatuiSato:dr} (which makes crucial use of the classification results of Winter \cite{Winter:RR0drClass,Winter:localizing} and Lin-Niu \cite{LinNiu:Lifting}).

By R\o rdam \cite{Rordam:pirange}, the algebra $\mathcal F_{m_0,m_1}$ exists, and by Kirchberg and Phillips classification \cite{Phillips:piClass}, it is unique.
\end{remark}

Some of these algebras are important and/or already well-known:

\begin{enumerate}
\item
$\mathcal E_{1,1,1} \cong \mathcal Z$ (the Jiang-Su algebra \cite{JiangSu}).
\item
$\mathcal E_{n,1,1} \cong M_n$ (a UHF algebra) for any infinite supernatural number $n$.
\item $\mathcal F_{1,1} \cong \mathcal O_2$ (a Cuntz algebra \cite{Cuntz:On}).
\item
Letting $m$ be the product of all primes infinitely many times, the C*-algebra
$T:= \mathcal F_{1,m}$ has the property that, if $A$ is a Kirchberg algebra in the UCT class, then $A$ is stably isomorphic to $A \otimes T$ if and only if $K_*(A)$ is a torsion group.
This follows from the K\"unneth formula (see \eqref{eq:KunnethFormula} below) and Kirchberg-Phillips classification \cite{Kirchberg:piClass,Phillips:piClass}.
\end{enumerate}

Our main result is the following.

\begin{thm}
\label{thm:MainThm}
Let $A$ be a separable, unital, C*-algebra with strict comparison, in the UCT class, which is either infinite or quasidiagonal.
The following are equivalent.
\begin{enumerate}
\item $A$ has approximately inner flip;
\item $A \otimes A$ has approximately inner flip;
\item $A$ has asymptotically inner flip;
\item $A$ is simple, nuclear, has at most one trace and $K_*(A)$ (as an ungraded, unordered group) is isomorphic to one of $0$, $\Z$, $\Q_n$, $\Q_m/\Z$, or $\Q_n \oplus \Q_m/\Z$, where $n,m$ are supernatural numbers of infinite type and $m$ divides $n$;
\item $A$ is stably isomorphic to one of:
\begin{enumerate}
\item $\mathbb C$;
\item $\mathcal E_{n,m_0,m_1}$;
\item $\mathcal E_{n,m_0,m_1} \otimes \mathcal O_\infty$;
\item $\mathcal E_{n,m_0,m_1} \otimes \mathcal O^\infty$; or
\item $\mathcal F_{m_0,m_1}$,
\end{enumerate}
where in $\mathrm{(b)}$-$\mathrm{(e)}$, $n$, $m_0$, and $m_1$ are supernatural numbers of infinite type such that $m_0,m_1$ are coprime and $m_0m_1|n$.
\end{enumerate}
\end{thm}

\begin{remark}
\label{rmk:MainThmClass}
We call the C*-algebras satisfying the hypotheses and equivalent conditions of Theorem \ref{thm:MainThm} the separable unital \textbf{classifiable C*-algebras with approximately inner flip}.
Understanding that ``classifiable'' means ``classifiable by K-theory and traces,'' then the classification results to date (including Kirchberg-Phillips' classification of purely infinite C*-algebras \cite{Kirchberg:piClass,Phillips:piClass} and the Gong-Lin-Niu classification of C*-algebras of generalized tracial rank one \cite{GongLinNiu}) permit an extremely reasonable definition of a simple separable unital ``classifiable C*-algebra'' as meaning a simple separable unital C*-algebra which satisfies the UCT and is either purely infinite or has generalized tracial rank one (in the sense of \cite[Definition 9.2]{GongLinNiu}).
Certainly, the aforementioned classification results show that the class of such C*-algebras is classifiable and exhausts the range of the Elliott invariant; hence, this is a maximal classifiable class.
The C*-algebras satisfying the hypotheses and equivalent conditions of Theorem \ref{thm:MainThm} are precisely the C*-algebras in this classifiable class which have approximately inner flip (this is entailed by the theorem).

It is a long-standing open question whether there are simple nuclear C*-algebras that (i) don't satisfy the UCT or (ii) are stably finite but not quasidiagonal.
Even for the much smaller and deeply studied class of strongly self-absorbing C*-algebras, this question is open.

However, unlike the class of strongly self-absorbing C*-algebras, it is unknown whether there exists a C*-algebra with approximately inner flip which does not have strict comparison (equivalently, whether it is $\mathcal Z$-stable, by \cite{MatuiSato:comp}).
\end{remark}

\begin{question}
(i)
Does there exist a C*-algebra with approximately inner flip which is not $\mathcal Z$-stable?

(ii)
Does there exist a C*-algebra $A$ such that $A \otimes A$ is strongly self-absorbing, but $A$ is not $\mathcal Z$-stable?
\end{question}

A positive answer to (ii) would imply a positive answer to (i), since if $A \otimes A$ is strongly self-absorbing then all of its automorphisms are approximately inner.

\section{The flip map on \texorpdfstring{$\Tor(G,G)$}{Tor(G,G)}}
\label{sec:TorFlipMap}

Let $G_1,G_2$ be abelian groups.
The flip isomorphism $\sigma_{G_1,G_2}: G_1 \otimes G_2 \to G_2 \otimes G_1$ induces a natural isomorphism
\[ \eta_{G_1,G_2}:\Tor(G_1,G_2) \to \Tor(G_2,G_1). \]

Here is a description of $\eta_{G_1,G_2}$.
Fix a free abelian group $F_i$ that surjects onto $G_i$; the kernel $H_i$ of this surjection is also a free group, and we get an exact sequence
\[ \xymatrix{
0 \ar[r] & H_i \ar[r] & F_i \ar[r] & G_i \ar[r] & 0 } \]
(called a \textbf{free resolution} of $G_i$.)

This induces a double-complex with exact rows and columns:

\begin{equation}
\label{eq:etaDoubleComplex}
 \xymatrix{
&&&& 0 \ar[d] & \\
&& 0 \ar[d] & 0 \ar[d] & \Tor(G_1, G_2) \ar[d] & \\
&0 \ar[r] & H_1 \otimes H_2 \ar[r]^-{\alpha_{11}} \ar[d]^-{\beta_{11}} & H_1 \otimes F_2 \ar[r]^-{\alpha_{12}} \ar[d]^-{\beta_{12}} & H_1 \otimes G_2 \ar[r] \ar[d]^-{\beta_{13}} & 0 \\
&0 \ar[r] & F_1 \otimes H_2 \ar[r]^-{\alpha_{21}} \ar[d]^-{\beta_{21}} & F_1 \otimes F_2 \ar[r]^-{\alpha_{22}} \ar[d]^-{\beta_{22}} & F_1 \otimes G_2 \ar[r] \ar[d]^-{\beta_{23}} & 0 \\
0 \ar[r] & \Tor(G_2,G_1) \ar[r] & G_1 \otimes H_2 \ar[r]^-{\alpha_{31}} \ar[d] & G_1 \otimes F_2 \ar[r]^-{\alpha_{32}} \ar[d] & G_1 \otimes G_2 \ar[r]\ar[d] & 0 \\
&&0&0&0.&
} \end{equation}
In particular, we identify $\Tor(G_1,G_2)$ and $\Tor(G_2,G_1)$ with $\ker(\beta_{13})$ and $\ker(\alpha_{31})$ respectively.
Let $x \in \Tor(G_1,G_2) = \ker(\beta_{13})$.
Let $x_{12} \in H_1 \otimes F_2$ be such that $\alpha_{12}(x_{12}) = x$.
Then
\[
\alpha_{22}\circ \beta_{12}(x_{12}) = \beta_{13} \circ \alpha_{21}(x_{12}) = \beta_{13}(x) = 0,
\]
i.e., $\beta_{12}(x_{12}) \in \ker(\alpha_{22})$.
By exactness of the second row, this means that $\beta_{12}(x_{12}) = \alpha_{21}(x_{21})$ for some unique $x_{21} \in F_1 \otimes H_2$.
Set $y:=\beta_{21}(x_{21})$.
Then
\[
\alpha_{31}(y) = \alpha_{31} \circ \beta_{21}(x_{21}) = \beta_{22} \circ \alpha_{21}(x_{21}) = 0,
\]
i.e., $y \in \ker(\alpha_{31}) = \Tor(G_2,G_1)$.
The element $y$ does not depend on the choice of $x_{12}$, and we have
\[ \eta_{G_1,G_2}(x) = y. \]

\section{The flip map and the K\"unneth formula}
\label{sec:KunnethFlip}

If $A$ is a C*-algebra in the UCT class then Schochet's K\"unneth Theorem \cite{Schochet:KunnethThm} provides an exact sequence for computing $K_*(A\otimes B)$.
We will use the statement and exposition in \cite[Chapter 23]{Blackadar:Kbook} (specifically, \cite[Theorem 23.1.3]{Blackadar:Kbook}); the exact sequence is
\begin{equation}
\label{eq:KunnethFormula}
\xymatrix{
0 \ar[r] & K_*(A) \otimes K_*(B) \ar[r]^-{\alpha_{A,B}} & K_*(A\otimes B) \ar[r]^-{\beta_{A,B}} & \Tor(K_*(A), K_*(B)) \ar[r] & 0,}
\end{equation}
where the maps involved are all natural, with $\alpha_{A,B}$ preserving the $\Z_2$-grading and $\beta_{A,B}$ reversing the $\Z_2$-grading.

In this section we prove the following.

\begin{lemma}
\label{lem:KunnethFlip}
The following commutes.
\begin{equation}
\label{eq:KunnethFlipDiag}
\xymatrix{
0 \ar[r] & K_*(A) \otimes K_*(B) \ar[r]^-{\alpha_{A,B}}\ar[d]_-{\sigma_{K_*(A),K_*(B)}} & K_*(A\otimes B) \ar[r]^-{\beta_{A,B}}\ar[d]_-{K_*(\sigma_{A,B})} & \Tor(K_*(A), K_*(B)) \ar[r]\ar[d]_-{\eta_{K_*(A),K_*(B)}} & 0 \\
0 \ar[r] & K_*(B) \otimes K_*(A) \ar[r]^-{\alpha_{B,A}} & K_*(B\otimes A) \ar[r]^-{\beta_{B,A}} & \Tor(K_*(B), K_*(A)) \ar[r] & 0.
} \end{equation}
\end{lemma}

For topological K-theory (equivalently, C*-algebra K-theory restricted to the class of commutative C*-algebras), a K\"unneth spectral sequence argument could be used to prove the above lemma, although the author was unable to find a precise reference for such an argument.
To the author's knowledge, the literature does not contain a proof of the above lemma for topological K-theory which extends easily to the case of C*-algebra K-theory.

That the left square in \eqref{eq:KunnethFlipDiag} commutes is well-known and not difficult.
Although commutativitiy of the right square seems quite natural, it is not trivial to show and the proof requires some setup.
Note that there are (at least) two natural isomorphisms $\Tor(G,H) \to \Tor(H,G)$, namely $\eta_{G,H}$ and $-\eta_{G,H}$; although one expects that one of these should fit into the commuting diagram \eqref{eq:KunnethFlipDiag}, it is not a priori obvious which one.

The following setup comes from \cite[Section 23.5]{Blackadar:Kbook}.
Since $K_*(A)$ is naturally isomorphic to $K_*(S^2A \otimes \mathcal K)$, we may assume (when proving Lemma \ref{lem:KunnethFlip}) that $A$ is of the form $S^2A' \otimes \mathcal K$ for some C*-algebra $A'$, and likewise for $B$.
Under this assumption, by \cite[Proposition 23.5.1]{Blackadar:Kbook}, there exists a separable commutative C*-algebra $F_A$, whose spectrum consists of disjoint union of lines and planes, and a homomorphism $\phi_A:F_A \to B$ giving rise to a surjective map $K_*(F_A) \to K_*(A)$.
With $C_A$ the mapping cone of this homomorphism, i.e.,
\[ C_A := \{(f,g) \in F_A \oplus C_0((0,1],A) \mid \phi(f)=g(1)\}, \]
we obtain an exact sequence
\[ \xymatrix{
0 \ar[r] & SA \ar[r]^-{\mu_A} & C_A \ar[r]^-{\nu_A} &F_A \ar[r] & 0,
} \]
whose 6-term exact sequence in K-theory becomes two short exact sequences,
\begin{equation}
\label{eq:ResolutionK}
 \xymatrix{
0 \ar[r] & K_*(C_A) \ar[r] & K_*(F_A) \ar[r]^-{\partial} &K_*(SA) \ar[r] & 0
} \end{equation}

Since $K_*(A) \cong K_*(SA)$ (by an isomorphism that reverses the grading), we may identify $\Tor(K_*(A),K_*(B))$ with $\Tor(K_*(SA),K_*(SB))$.
In turn, we identify $\Tor(K_*(SA),K_*(SB))$ with the kernel of 
\[ (\nu_A)_* \otimes 1_{K_*(SB)}:K_*(C_A) \otimes K_*(SB) \to K_*(F_A) \otimes K_*(SB); \]
since $K_*(F_A)$ and $K_*(C_A)$ are free abelian, we have by the K\"unneth formula \eqref{eq:KunnethFormula}, a commuting diagram as follows:
\[ \xymatrix{
K_*(C_A) \otimes K_*(SB) \ar[r]^-{\cong} \ar[d]_-{(\nu_A)_* \otimes 1_{K_*(SB)}} &
K_*(C_A\otimes SB) \ar[d]^-{(\nu_A \otimes \id_{SB})_*} \\
K_*(F_A) \otimes K_*(SB) \ar[r]^-{\cong} & K_*(F_A \otimes SB).
} \]
Thus, we actually identify $\Tor(K_*(A),K_*(B))$ with the kernel of $(\nu_{A} \otimes \id_{SB})_*$.
Under this identification, $\beta_{A,B}:K_*(A\otimes B) \cong K_*(SA \otimes SB) \to \Tor(K_*(A),K_*(B))$ is precisely the map 
\[ (\id_{SA} \otimes \mu_B)_*:K_*(SA \otimes SB) \to K_*(SA \otimes C_B), \]
(by a 6-term exact sequence, the image of this map is indeed contained in the kernel of $(\id_{SA} \otimes \nu_B)_*$) (see the proof of \cite[Proposition 23.6.1]{Blackadar:Kbook}).

By the same construction with $B$ in place of $A$, obtain $F_B$, $C_B$, $ \mu_B$, and $ \nu_B$.

\begin{lemma}
\label{lem:iotaKinj}
Let $C_A,C_B$ be as described above.
Define $\iota_1:SA \otimes C_B \to SA \otimes C_B + C_A \otimes SB$ to be the inclusion.
Then $(\iota_1)_*:K_*(SA \otimes C_B) \to K_*(SA \otimes C_B + C_A \otimes SB)$ is injective.
\end{lemma}

\begin{proof}
Consider the following commuting diagram with short exact rows:
\begin{equation}
\label{eq:iotaKinjExactSeqs}
 \xymatrix{
0 \ar[r] & SA \otimes C_B \ar[r]^-{\iota_1} \ar[d]_-{\id} & SA \otimes C_B + C_A \otimes SB \ar[r] \ar[d] & F_A \otimes SB \ar[r] \ar[d]^-{\id_{F_A} \otimes \mu_B} & 0\\
0 \ar[r] & SA \otimes C_B \ar[r]^-{\mu_A \otimes \id_{C_B}} & C_A \otimes C_B \ar[r]^-{\nu_A \otimes \id_{C_B}} & F_A \otimes C_B \ar[r] & 0.
}
\end{equation}
The top row produces the 6-term exact sequence
\[ \xymatrix{
K_0(SA \otimes C_B) \ar[r]^-{(\iota_1)_*} & K_0(SA \otimes C_B + C_A \otimes SB) \ar[r] & K_0(F_A \otimes SB) \ar[d]^-{\partial_0} \\
K_1(F_A \otimes SB) \ar[u]_-{\partial_1} & K_1(SA \otimes C_B + C_A \otimes SB) \ar[l] & K_1(SA \otimes C_B). \ar[l]^-{(\iota_1)_*},
} \]
so that injectivity of $(\iota_1)_*$ is equivalent to $\partial_*:K_*(SA \otimes F_B) \to K_{*}(C_A \otimes SB)$ being the zero map.  

The second row of \eqref{eq:iotaKinjExactSeqs} produces a 6-term exact sequence which, by \eqref{eq:ResolutionK} and the K\"unneth formula \eqref{eq:KunnethFormula} (since $K_*(C_B)$ is a free abelian group), becomes two short exact sequences,
\[ \xymatrix{
0 \ar[r] & K_i(C_A \otimes C_B) \ar[r]^-{(\nu_A \otimes \id_{C_B})_*} & K_i(F_A \otimes C_B) \ar[r]^-{\partial \otimes 1} &K_{1-i}(SA \otimes C_B) \ar[r] & 0,
} \]
$i=0,1$.

By \eqref{eq:iotaKinjExactSeqs} and naturality of the 6-term exact sequence, the following commutes
\begin{equation} 
\label{eq:iotaKinjCommKDiag}
\xymatrix{
K_*(F_A \otimes SB) \ar[r]^-{\partial} \ar[d]_-{(\id_{F_A} \otimes \mu_B)_*} & K_*(SA \otimes C_B) \ar[d]^-{1} \\
K_*(F_A \otimes C_B) \ar[r]^-{\partial \otimes 1} & K_*(SA \otimes C_B).
} \end{equation}

Using \eqref{eq:ResolutionK} (with $B$ in place of $A$), and the K\"unneth formula \eqref{eq:KunnethFormula} (since $K_*(F_A)$ is a free abelian group), the map
\[ (\id_{F_A} \otimes \mu_B)_*:K_*(F_A \otimes SB) \to K_*(F_A \otimes C_B) \]
is zero.
Thus, by \eqref{eq:iotaKinjCommKDiag}, the map $\partial:K_*(F_A \otimes SB) \to K_*(SA \otimes C_B)$ is zero, as required.
\end{proof}

\begin{proof}[Proof of Lemma \ref{lem:KunnethFlip}]
The map $\alpha_{A,B}$ is explicitly described in \cite[Section 23.1]{Blackadar:Kbook}, and it is apparent from this description that $\alpha_{B,A} \circ \sigma_{K_*(A),K_*(B)} = K_*(\sigma_{A,B}) \circ \alpha_{A,B}$, i.e., the first square in \eqref{eq:KunnethFlipDiag} commutes.
Let us move on to the second square.

With $G_1:=K_*(A)$ and $G_2:=K_*(B)$, we use the description of $\eta_{K_*(A),K_*(B)}$ from Section \ref{sec:TorFlipMap}, making use of free resolutions of $K_*(A)$, $K_*(B)$ provided by \eqref{eq:ResolutionK} (we assume that these exist by possibly replacing $A$, $B$ by $S^2A \otimes \mathcal K$, $S^2B \otimes \mathcal K$ respectively).
Note that, since $K_*(F_A)$, $K_*(F_B)$, $K_*(C_A)$, and $K_*(C_B)$ are all free abelian groups, the K\"unneth formula \eqref{eq:KunnethFormula} turns the double complex \eqref{eq:etaDoubleComplex} into
\[
\scriptsize{
 \xymatrix{
&&& \Tor(K_*(A), K_*(B)) \ar[d]  \\
& K_*(C_A \otimes C_B) \ar[r] \ar[d] & K_*(C_A \otimes F_B) \ar[r]^-{1 \otimes \partial} \ar[d]^-{(\nu_A \otimes \id_{F_B})_*} & K_*(C_A \otimes SB) \ar[d] \\
& K_*(F_A \otimes C_B) \ar[r]^-{(\id_{F_A} \otimes \nu_B)_*} \ar[d]^-{\partial \otimes 1} & K_*(F_A \otimes F_B) \ar[r] \ar[d] & K_*(F_A \otimes SB) \ar[d]\\
\Tor(K_*(B),K_*(A)) \ar[r] & K_*(SA \otimes C_B) \ar[r] & K_*(SA \otimes F_B) \ar[r] & K_*(SA) \otimes K_*(SB)\\
}} \]
(for space considerations, the zero terms are omitted).

Now, let $x \in K_*(SA \otimes SB) \cong K_*(A,B)$.
We have $\beta_{A,B}(x)=(\id_{SA} \otimes \nu_B)_*(x) \in \Tor(K_*(A),K_*(B))$.
Using the description of $\eta$ from Section \ref{sec:TorFlipMap}, there exist $x_{12} \in K_*(C_A\otimes F_B)$ and $x_{21} \in K_*(F_A\otimes C_B)$ such that
\begin{align}
\label{eq:KunnethFlipSetup1}
(1 \otimes \partial)(x_{12}) &= \beta_{A,B}(x) = (\id_{SA} \otimes \nu_B)_*(x), \\
\label{eq:KunnethFlipSetup2}
(\nu_A \otimes \id_{F_B})_*(x_{12}) &= (\id_{F_A} \otimes \nu_B)_*(x_{21}), \quad \text{and} \\
\label{eq:KunnethFlipSetup3}
(\partial \otimes 1)(x_{21}) &= \eta_{K_*(A),K_*(B)}(\beta_{A,B}(x)).
\end{align}

Consider the following commuting diagram with short exact rows:
\[ \xymatrix{
0 \ar[r] & SA \otimes C_B \ar[r] \ar[d]_-{\iota_1} & C_A \otimes C_B \ar[r] \ar[d] & F_A \otimes C_B \ar[r] \ar[d]^-{\id_{F_A} \otimes \nu_B} & 0\\
0 \ar[r] & SA \otimes C_B+C_A \otimes SB \ar[r] & C_A \otimes C_B \ar[r] & F_A \otimes F_B \ar[r] & 0.
} \]
By naturality of the 6-term exact sequences, we obtain the following commuting diagram:
\begin{equation}
\label{eq:KunnethFlipCommuting1}
 \xymatrix{
K_*(F_A \otimes C_B) \ar[d]_-{(\id_{F_A} \otimes \nu_B)_*} \ar[r]^-{\partial \otimes 1} & K_*(SA \otimes C_B) \ar[d]^-{(\iota_1)_*} \\
K_*(F_A \otimes F_B) \ar[r]^-{\partial}  & K_*(SA \otimes C_B+C_A\otimes SB).
} \end{equation}
Likewise, the following also commutes
\begin{equation}
\label{eq:KunnethFlipCommuting2}
 \xymatrix{
K_*(C_A \otimes F_B) \ar[d]_-{(\nu_A \otimes \id_{F_B})_*} \ar[r]^-{1\otimes\partial} & K_*(C_A \otimes SB) \ar[d]^-{(\iota_2)_*} \\
K_*(F_A \otimes F_B) \ar[r]^-{\partial}  & K_*(SA \otimes C_B+C_A\otimes SB),
} \end{equation}
where $\iota_2$ denotes the inclusion $C_A \otimes SB \to SA \otimes C_B + C_A \otimes SB$.

Note that
\begin{equation}
\label{eq:KunnethFlipInclusions}
\iota_1 \circ (\nu_A \otimes \id_{SB})=\iota_2\circ (\id_{SA} \otimes \nu_B):SA \otimes SB \to SA \otimes C_B + C_A \otimes SB,
\end{equation}
both maps being equal to the inclusion.

Putting these pieces together, we obtain
\begin{eqnarray*}
(\iota_1)_* \circ \eta_{K_*(A),K_*(B)} \circ \beta_{A,B}(x) &\stackrel{\eqref{eq:KunnethFlipSetup3}}=& (\iota_1)_* \circ (\partial \otimes 1)(x_{21}) \\
&\stackrel{\eqref{eq:KunnethFlipCommuting1}}=& \partial \circ (\id_{F_A} \otimes \nu_B)_*(x_{21}) \\
&\stackrel{\eqref{eq:KunnethFlipSetup2}}=& \partial \circ (\nu_A \otimes \id_{F_B})_*(x_{12}) \\
&\stackrel{\eqref{eq:KunnethFlipCommuting2}}=& (\iota_2)_* \circ (1 \otimes \partial)(x_{12}) \\
&\stackrel{\eqref{eq:KunnethFlipSetup1}}=& (\iota_2)_* \circ (\id_{SA} \otimes \nu_B)_*(x) \\
&\stackrel{\eqref{eq:KunnethFlipInclusions}}=& (\iota_1)_* \circ (\nu_A \otimes \id_{SB})_*(x).
\end{eqnarray*}
By Lemma \ref{lem:iotaKinj}, it follows that
\[ \eta_{K_*(A),K_*(B)} \circ \beta_{A,B}(x) = (\nu_A \circ \id_{SB})_*(x) = \beta_{B,A}\circ (\sigma_{A,B})_*(x), \]
as required.
\end{proof}

\section{Sufficient conditions for approximately inner flip}
\label{sec:Suff}

\begin{prop}
\label{prop:PruferFlip}
Let $n$ be a supernatural number.
Then $\eta_{\Q_n/\Z,\Q_n/\Z}$ is the identity map.
\end{prop}

\begin{proof}
For each natural number $m|n$, $\Q_m/\Z \subset \Q_n/\Z$, and this inclusion (used twice) produces a commutative diagram
\[ \xymatrix{
\Tor(\Q_m/\Z, \Q_m/\Z) \ar[r]^-{\eta_{\Q_m/\Z,\Q_m/\Z}} \ar[d] & \Tor(\Q_m/\Z, \Q_m/\Z) \ar[d] \\
\Tor(\Q_n/\Z, \Q_n/\Z) \ar[r]^-{\eta_{\Q_n/\Z,\Q_n/\Z}} & \Tor(\Q_n/\Z, \Q_n/\Z). } \] 
We have $\Tor(\Q_m/\Z, \Q_m/\Z) \cong \Q_m/\Z$ and $\Tor(\Q_n/\Z,\Q_n/\Z) \cong \Q_n/\Z$ by \cite[62.J]{Fuchs:book}, and these isomorphisms induce the natural inclusion $\Q_m/\Z \subset \Q_n/\Z$.
Since $\Q_n/\Z$ is the union of such subgroups, it suffices to show that $\eta_{\Q_m/\Z, \Q_m/\Z}$ is the identity map.

For this, note that $\Q_m \cong \Z$, so we have a free resolution
\[ 0 \to \Z \to \Q_m \to \Q_m/\Z \to 0. \]
Set $H:=\Z, F:=\Q_m, G:=\Q_m/\Z$ and refer to the description of $\eta_{\Q_m/\Z,\Q_m/\Z}$ in Section \ref{sec:TorFlipMap}.
Let $x \in \Tor(\Q_m/\Z,\Q_m/\Z)$, which we identify with $\Z \otimes (\Q_m/\Z)$ (because the map $\Z \otimes (\Q_m/\Z) \to \Q_m \otimes (\Q_m/\Z)$ is the zero map); thereby write $x=1 \otimes (k/m + \Z)$ for some $k \in \Z$.
This lifts to $x_{12} = 1 \otimes (k/m) \in \Z \otimes \Q_m$, which is equal to $(k/m) \otimes 1$ in $\Q_m \otimes \Q_m$.
Thus, $x_{21} = (k/m) \otimes 1$, so that
\[ \eta_{\Q_m/\Z,\Q_m/\Z}(x) = (k/m + \Z) \otimes 1. \]
This establishes that $\eta_{\Q_m/\Z,\Q_m/\Z}$ is the identity map.
\end{proof}

\begin{thm}
\label{thm:KKsuff}
Let $A$ be a separable $C^*$-algebra.
Suppose that $K_*(A)$ is one of the following groups (ignoring the grading):
\begin{enumerate}
\item $0$;
\item $\Z$;
\item $\Q_n$, where $n$ is a supernatural number of infinite type;
\item $\Q_m/\Z$ where $m$ is a supernatural number of infinite type; or
\item $\Q_n \oplus \Q_m/\Z$, where $n,m$ are supernatural numbers of infinite type, and $m$ divides $n$.
\end{enumerate}
Then the flip map $\sigma_{A,A}:A \otimes A \to A \otimes A$ has the same KK-class as the identity map.
\end{thm}

\begin{proof}
Using the K\"unneth formula \eqref{eq:KunnethFormula}, one computes $K_*(A \otimes A)$:
\begin{enumerate}
\item $K_*(A \otimes A)=0$;
\item $K_0(A\otimes A)=\Z$ and $K_1(A\otimes A)=0$;
\item $K_0(A \otimes A)=\Q_n$ and $K_1(A\otimes A)=0$;
\item $K_0(A\otimes A)=0$ and $K_1(A\otimes A)=\Q_m/\Z$;
\item $K_0(A\otimes A)=\Q_n$ and $K_1(A\otimes A)=\Q_m/\Z$.
\end{enumerate}
By Lemma \ref{lem:KunnethFlip} and (for cases (iv) and (v)) Proposition \ref{prop:PruferFlip}, the flip map $\sigma_{A,A}$ acts as the identity on $K_*(A\otimes A)$.

In cases (i)-(iv), $\Ext(K_i(A \otimes A), K_{1-i}(A \otimes A))=0$ for trivial reasons, for $i=0,1$.
In case (v), $\Ext(K_i(A\otimes A), K_{1-i}(A\otimes A))=0$ by Lemma \ref{lem:DivSplit}.
Thus, by the UCT \cite[Theorem 23.1.1]{Blackadar:Kbook}, it follows that $\sigma_{A,A}$ agrees with the identity in KK.
\end{proof}

Using classification, we obtain the following, 

\begin{cor}
\label{cor:Suff}
Let $m_0,m_1,n$ be supernatural numbers of infinite type, such that $m_0,m_1$ are coprime and $m_0m_1$ divides $n$.
Then $\mathcal E_{n,m_0,m_1}$ and $\mathcal F_{m_0,m_1}$ have asymptotically inner flip.
\end{cor}

\begin{proof}
The algebras $\mathcal E_{n,m_0,m_1}$ and $\mathcal F_{m_0,m_1}$ satisfy the K-theoretic hypotheses of Theorem \ref{thm:KKsuff}, and therefore the flip map has the same KK-class as the identity.
For $\mathcal F_{m_0,m_1}$, it follows from Kirchberg-Phillips classification (see \cite[Theorem 8.2.1 (ii) or 8.3.3 (iii)]{Rordam:ClassBook}) that the flip is asymptotically inner.

For $A:=\mathcal E_{n,m_0,m_1}$ (where $n\neq 1$), first note that $A$ is a simple AH algebra with real rank zero, by \cite[Theorem 4.18]{ElliottGong:RR0II} (see also Remark \ref{rmk:EFdef}); also, the $K_0$-group is unperforated.
Hence by \cite[Theorem 2.1]{Lin:AHRR0}, it has tracial rank zero.
We shall appeal to Lin's result \cite[Theorem 10.7]{Lin:AsympEquiv}, which says that $\sigma_{A,A}$ and $\id_{A \otimes A}$ are asymptotically unitarily equivalent provided that they have the same KK-class, they agree on traces, and the rotation map $\tilde\eta_{\sigma_{A,A},\id_{A \otimes A}}$ vanishes.
The first hypothesis has already been verified; the latter two hypotheses are true for trivial reasons, as follows.
Since $A \otimes A$ has unique trace, all automorphisms must agree on this trace.
The condition $\tilde\eta_{\sigma_{A,A},id_{A \otimes A}}=0$ means that a certain map $K_1(A \otimes A) \to \mathrm{Aff}(T(A \otimes A))$ is the zero map (see \cite[Definition 3.4]{Lin:AsympEquiv}); since $K_1(A \otimes A)$ is a torsion group, this holds automatically.

Finally, in the case $n=1$, $A=\mathcal E_{1,1,1} \cong \mathcal Z$, which has asymptotically inner flip by \cite[Theorem 2.2]{DadarlatWinter:KKssa}.
\end{proof}

\section{Necessary conditions}
\label{sec:Necessary}

In this section, we shall prove (ii) $\Rightarrow$ (iv) of Theorem \ref{thm:MainThm}.
In fact, this implication is proven in potentially weaker generality (not assuming the conditions of strict comparison and infinite or quasidiagonal), as follows:

\begin{thm}
\label{thm:Necessary}
Let $A$ be a $C^*$-algebra in the UCT class, such that $A \otimes A$ has approximately inner flip.
Then $K_*(A)$ is isomorphic to one of the following groups (ignoring the grading):
\begin{enumerate}
\item $0$;
\item $\Z$;
\item $\Q_n$, where $n$ is a supernatural number of infinite type;
\item $\Q_m/\Z$ where $m$ is a supernatural number of infinite type; or
\item $\Q_n \oplus \Q_m/\Z$, where $n,m$ are supernatural numbers of infinite type, and $m$ divides $n$.
\end{enumerate}
\end{thm}

This result is derived from the following, a simple consequence of Lemma \ref{lem:KunnethFlip}.

\begin{lemma}
\label{lem:BasicRestrictions}
Let $A$ be a C*-algebra in the UCT class with approximately inner flip.
If $K_*(A)$ contains direct summands $G_1$ and $G_2$ (irrespective of the grading) then:
\begin{enumerate}
\item $G_1 \otimes G_2 = 0$; and
\item $\Tor(G_1,G_2)=0$.
\end{enumerate}
\end{lemma}

\begin{proof}
Since the flip $\sigma_{A,A}:A\otimes A \to A\otimes A$ is approximately inner, it must agree with the identity map on K-theory.
By Lemma \ref{lem:KunnethFlip}, this implies that $\sigma_{K_*(A),K_*(A)}$ and $\eta_{K_*(A),K_*(A)}$ are both the identity map.

(i):
Let $K_*(A)=G_1 \oplus G_2 \oplus G_3$.
We have
\[ K_*(A) \otimes K_*(A) = (G_1 \otimes G_1) \oplus (G_1 \otimes G_2) \oplus (G_2 \otimes G_1) \oplus (G_2 \otimes G_2) \oplus H, \]
(where $H$ involves $G_3$) and $\sigma_{K_*(A),K_*(A)}$ sends $G_1 \otimes G_2$ to $G_2 \otimes G_1$ (by the flip isomorphism).
Therefore, if $G_1 \otimes G_2 \neq 0$, then $\sigma_{K_*(A),K_*(A)}$ cannot be the identity map, which is a contradiction.

(ii) is essentially the same argument: we have
\begin{align*}
 \Tor(K_*(A),K_*(A)) =  &\Tor(G_1,G_1) \oplus \Tor(G_1,G_2) \oplus \Tor(G_2,G_1) \\
&\qquad \oplus \Tor(G_2,G_2) \oplus H',
\end{align*}
and $\eta_{K_*(A),K_*(A)}$ sends $\Tor(G_1,G_2)$ to $\Tor(G_2,G_1)$ (by $\eta_{G_1,G_2}$).
\end{proof}

\begin{lemma}
\label{lem:TensorKComp}
Let $A$ be a $C^*$-algebra in the UCT class which has approximately inner flip.
Then 
\begin{align*}
K_0(A \otimes A) &\cong \big(K_0(A) \otimes K_0(A)\big) \oplus \big(K_1(A) \otimes K_1(A)\big) \quad \text{and} \\
K_1(A \otimes A) &\cong \Tor(K_0(A),K_0(A)) \oplus \Tor(K_1(A),K_1(A)).
\end{align*}
\end{lemma}

\begin{proof}
For C*-algebras $A,B$, the exact sequence from the K\"unneth formula \eqref{eq:KunnethFormula} for $A \otimes B$ can be expressed as two exact sequences,
\begin{align*}
0 &\to K_0(A) \otimes K_0(B) \oplus K_1(A) \otimes K_1(B) \to K_0(A \otimes B) \\
&\qquad \to \Tor(K_0(A),K_1(B)) \oplus \Tor(K_1(A),K_0(B)) \to 0, \\
0 &\to K_0(A) \otimes K_1(B) \oplus K_0(A) \otimes K_1(B) \to K_1(A \otimes B) \\
&\qquad \to \Tor(K_0(A),K_0(B)) \oplus \Tor(K_1(A),K_1(B)) \to 0.
\end{align*}

By Lemma \ref{lem:BasicRestrictions}, $K_0(A) \otimes K_1(A) = 0$ and $\Tor(K_0(A),K_1(A)) = 0$.
Putting these together yield the result.
\end{proof}

\begin{lemma}
\label{lem:NoCycSummand}
Let $A$ be a C*-algebra in the UCT class and let $G_p$ be a direct summand of $K_*(A)$ which is a nonzero $p$-group for some prime $p$.
Suppose that
\begin{enumerate}
\item $A$ has approximately inner flip; or
\item If $A \otimes A$ has approximately inner flip and $K_*(A)=G_p$.
\end{enumerate}
Then $G_p \cong \Q_{p^\infty}/\Z$.
\end{lemma}

\begin{proof}
(i):
We assume that $A$ has approximately inner flip.
First let us show that $G_p$ is directly indecomposable, i.e., that it cannot be expressed as a direct sum of two nontrivial groups.
If $G_p=H_1 \oplus H_2$ and $H_1,H_2$ are both nonzero then both are $p$-groups.
Hence $H_1$ contains a subgroup isomorphic to $\Z/p\Z$ and so, by the Cartan-Eilenberg exact sequence for Tor, $\Tor(H_1,H_2)$ contains a subgroup isomorphic to $\Tor(\Z/p\Z,H_2) \neq 0$.
By Lemma \ref{lem:BasicRestrictions} (ii), this contradicts that $A$ has approximately inner flip.

Therefore, $G_p$ is directly indecomposable and by \cite[Corollary 27.4]{Fuchs:book}, $G_p$ is either cyclic or isomorphic to $\Q_{p^\infty}/\Z$.

To rule out the case that $G_p$ is cyclic, suppose for a contradiction that $G_p\cong \Z/n\Z$ ($n$ is necessarily a power of $p$).
Since $G_p$ is directly indecomposable, it occurs in $K_i(A)$ for either $i=0$ or $1$.
Then we see that $\Z/n\Z$ is a direct summand of both $K_i(A) \otimes K_i(A)$ and of $\Tor(K_i(A),K_i(A))$.
By Lemma \ref{lem:TensorKComp}, $\Z/n\Z$ is a direct summand of both $K_0(A \otimes A)$ and $K_1(A \otimes A)$, so that $K_*(A \otimes A)$ contains two direct summands isomorphic to $\Z/n\Z$.
Since $(\Z/n\Z) \otimes (\Z/n\Z) \neq 0$ and $A \otimes A$ has approximately inner flip, this contradicts Lemma \ref{lem:BasicRestrictions} (i).

(ii):
We now assume that $A \otimes A$ has approximately inner flip and that $K_*(A)=G_p$.
By the K\"unneth formula \eqref{eq:KunnethFormula}, $K_*(A \otimes A) \cong \Tor(K_*(A),K_*(A))$.
By the same argument used in (i) to show that $G_p$ is directly indecomposable, if $K_*(A)$ is not directly indecomposable, then neither is $\Tor(K_*(A),K_*(A))$, which cannot be the case by (i).
Hence $G_p$ is directly indecomposable, so by \cite[Corollary 27.4]{Fuchs:book}, $G_p$ is either cyclic or isomorphic to $\Q_{p^\infty}/\Z$
Since $\Tor(K_*(A),K_*(A)) \cong \Q_{p^\infty}/\Z$ (by (i)), only the latter case is possible.
\end{proof}

\begin{proof}[Proof of Theorem \ref{thm:Necessary}]
Let $G=K_*(A)$ and let $T_G$ denote the torsion subgroup of $G$.
Our steps are as follows:
\begin{enumerate}
\item $G/T_G$ has rank at most one;
\item $G$ splits as a direct sum $T_G \oplus G/T_G$;
\item the theorem.
\end{enumerate}

(i): Let $\mathcal Q$ be the universal UHF algebra, so that $A \otimes A \otimes \mathcal Q$ has approximately inner flip.
By the K\"unneth formula \eqref{eq:KunnethFormula}, $K_*(A \otimes \mathcal Q) \cong G \otimes \Q$, which is a rational vector space, and likewise $K_*(A \otimes A \otimes \mathcal Q) \cong (G \otimes \Q)^{\otimes 2}$, so by Lemma \ref{lem:BasicRestrictions} (i), $G \otimes \Q$ must be either $0$ or $\Q$.
Note that $G \otimes \Q \cong (G/T_G) \otimes \Q$ so that $G/T_G$ has rank at most one.

(ii): Since $G/T_G$ is a torsion free group of rank at most one, it is either $0$ or a subgroup of $\Q$.
If $G/T_G=0$, there is nothing more to show for (ii).

If $G/T_G\cong \Z$ then the exact sequence $0 \to T_G \to G \to G/T_G \to 0$ splits and (ii) is established.

Otherwise, $G/T_G \cong \Q_n$ for some supernatural number $n$ of infinite type.
In order to use Lemma \ref{lem:DivSplit} to show that $G\cong T_G \oplus G/T_G$, we need to show that $T_G$ is $n$-divisible.
According to \cite[Theorem 8.4]{Fuchs:book}, $T_G$ is a direct sum of $p$-components $T_p$, over all primes $p$.
For $p$ coprime with $n$, $T_p$ is $n$-divisible; we need to show that $T_p$ is also $n$-divisible when $p$ divides $n$.

Let $p$ be a prime which divides $n$.
By Corollary \ref{cor:Suff}, $\mathcal F_{1,p^\infty}$ has approximately inner flip, whence so does $(A \otimes \mathcal F_{1,p^\infty})^{\otimes 2}$.
We see that $K_*(\mathcal F_{1,p^\infty})\otimes T_G=0$ (since $K_*(\mathcal F_{1,p^\infty})$ is divisible) and $K_*(\mathcal F_{1,p^\infty}) \otimes G/T_G=0$ (since $G/T_G\cong \Q_n$ is $p$-divisible).
Since
\[ K_*(\mathcal F_{1,p^\infty}) \otimes T_G \to K_*(\mathcal F_{1,p^\infty}) \otimes G \to K_*(\mathcal F_{1,p^\infty}) \otimes G/T_G \]
is exact, it follows that $K_*(\mathcal F_{1,p^\infty}) \otimes G = 0$.
Therefore, by the K\"unneth formula \eqref{eq:KunnethFormula}, 
\[ K_*(A \otimes \mathcal F_{1,p^\infty}) \cong \Tor(K_*(A),K_*(\mathcal F_{1,p^\infty})) \cong \Tor(K_*(A),\Q_{p^\infty}/\Z). \]
This is isomorphic to $T_p$ by \cite[62.J]{Fuchs:book}.
By Lemma \ref{lem:NoCycSummand} (ii) (applied to $A \otimes \mathcal F_{1,p^\infty}$ in place of $A$), $T_p$ is either $0$ or $\Q_{p^\infty}/\Z$; in either case, it is $p$-divisible.

This establishes that $T_G$ is $n$-divisible, so by Lemma \ref{lem:DivSplit}, $G\cong T_G \oplus G/T_G$.

(iii):
Now we know that $G\cong T_G \oplus G/T_G$, and $G/T_G$ has rank at most $1$.
If $G/T_G \cong \Z$ then by Lemma \ref{lem:BasicRestrictions}, $T_G=0$.

Otherwise, $G/T_G$ is either $\Q_n$ for some supernatural number $n$ of infinite type or $0$ (in which case we set $n=0$).
The summand $T_G$ is the direct sum of its $p$-components $T_p$ over all primes $p$, and the argument in (ii) shows that, for each $p|n$, $T_p$ is either $0$ or $\Q_{p^\infty}/\Z$.
For each prime $p$ that doesn't divide $n$ (in the case $G/T_G \cong \Q_n$), we have $\Q_n \otimes T_p \cong T_p$, and therefore by Lemma \ref{lem:BasicRestrictions}, $T_p = 0$.

Let $m$ be the supernatural number given by taking the product of all primes $p$ for which $T_p\neq 0$.
Thus, $m|n$ and $T_G \cong \Q_m/\Z$, so that 
\[ G \cong \begin{cases} \Q_n \oplus \Q_m/\Z, \quad &G/T_G \cong \Q_n; \\ \Q_m/\Z, \quad &G/T_G=0. \end{cases} \]
\end{proof}

\begin{proof}[Proof of Theorem \ref{thm:MainThm}]
(iii) $\Rightarrow$ (i) $\Rightarrow$ (ii) are immediate.

(ii) $\Rightarrow$ (iv): By \cite[Proposition 2.10]{EffrosRosenberg}, $A \otimes A$ has a unique trace, and therefore so does $A$.
Moreover, since $A$ has approximately inner half-flip, it is simple and nuclear \cite[Lemma 3.10]{KirchbergPhillips:ExactEmbedding} (alternatively, one could use \cite[Propositions 2.7, 2.8]{EffrosRosenberg} on $A\otimes A$ here).
The K-theoretic restrictions follow from Theorem \ref{thm:Necessary}.

(iv) $\Rightarrow$ (v):
There exists a C*-algebra $B$ isomorphic to one of $\mathcal E_{n,m_0,m_1}$, $E_{n,m_0,m_1} \otimes \mathcal O_\infty$, $E_{n,m_0,m_1} \otimes \mathcal O^\infty$, or $\mathcal F_{m_0,m_1}$, which has the same Elliott invariant (consisting of graded, ordered K-theory, the class of the unit in $K_0$, the space of traces, and the pairing between traces and $K_0$) as $A$, except possibly for the $K_0$-class of the unit.
If $A$ is type I then it must be stably isomorphic to $\mathbb C$.
Otherwise, it is $\mathcal Z$-stable (by \cite{MatuiSato:comp} in the finite case or \cite{KirchbergPhillips:ExactEmbedding} in the infinite case), and either quasidiagonal or purely infinite; in either case classification results show it is stably isomorphic to $B$ (see Remark \ref{rmk:EFdef})

Finally, (v) $\Rightarrow$ (iii) follows from Corollary \ref{cor:Suff}.
\end{proof}

\section{Semigroup structure under the tensor product operation}
\label{sec:Semigroup}

For a supernatural number $n$ of infinite type, let $P_n$ denote the set of all primes which divide $n$.

\begin{prop}
\label{prop:EtensE}
Let $n^{(1)},m^{(1)}_0,m^{(1)}_1,n^{(2)},m^{(2)}_0,m^{(2)}_1$ be supernatural numbers of infinite type, such that $m^{(i)}_0,m^{(i)}_1$ are coprime and $m^{(i)}_0m^{(i)}_1$ divides $n^{(i)}$ for $i=1,2$.
Then $\mathcal E_{n^{(1)},m^{(1)}_0,m^{(1)}_1} \otimes \mathcal E_{n^{(2)},m^{(2)}_0,m^{(2)}_1} \cong \mathcal E_{n,m_0,m_1}$ where:
\begin{enumerate}
\item $n:=n^{(1)}n^{(2)}$;
\item $m_0$ is the product of the following set of primes, each taken infinitely many times
\[ (P_{m^{(1)}_0} \setminus P_{n^{(2)}}) \cup (P_{m^{(2)}_0} \setminus P_{n^{(1)}}) \cup (P_{m^{(1)}_0} \cap P_{m^{(2)}_1}) \cup (P_{m^{(2)}_0} \cap P_{m^{(1)}_1}); \quad \text{and} \]
\item $m_1$ is the product of the following set of primes, each taken infinitely many times
\[ (P_{m^{(1)}_1} \setminus P_{n^{(2)}}) \cup (P_{m^{(2)}_1} \setminus P_{n^{(1)}}) \cup (P_{m^{(1)}_1} \cap P_{m^{(2)}_1}) \cup (P_{m^{(2)}_0} \cap P_{m^{(1)}_0}). \]
\end{enumerate}
\end{prop}

\begin{proof}
Set $A:=\mathcal E_{n^{(1)},m^{(1)}_0,m^{(1)}_1} \otimes \mathcal E_{n^{(2)},m^{(2)}_0,m^{(2)}_1}$.
This is a simple, separable, unital, $\mathcal Z$-stable, quasidiagonal C*-algebra that satisfies the UCT.
By the K\"unneth formula \eqref{eq:KunnethFormula}, we obtain the following short exact sequence
\begin{align}
\label{eq:EtensEexseq}
\notag
&0 \rightarrow (\Q_{n^{(1)}} \oplus \Q_{m^{(1)}_0}/\Z) \otimes (\Q_{n^{(2)}} \oplus \Q_{n^{(2)}}/\Z) \\
&\qquad \rightarrow K_0(A) \rightarrow \Tor(\Q_{m^{(1)}_0},\Q_{m^{(2)}_1}) \oplus \Tor(\Q_{m^{(1)}_1},\Q_{m^{(2)}_0}) \to 0.
\end{align}
The first term in this exact sequence is equal to
\[ \Q_{n^{(1)}n^{(2)}} \oplus \Q_{k_1^\infty}/\Z \oplus \Q_{k_2^\infty}/\Z \]
where $k_1$ is the product of the primes in $(P_{m^{(1)}_0} \setminus P_{n^{(2)}})$ and $k_2$ is the product of the primes in $(P_{m^{(2)}_0} \setminus P_{n^{(1)}})$.
Since this term is a divisible group, it is a direct summand of $K_0(A)$.

By \cite[62.J]{Fuchs:book}, the last term in the exact sequence \eqref{eq:EtensEexseq} is equal to
\[ \Q_{k_3^\infty}/\Z \oplus \Q_{k_4^\infty}/\Z, \]
where $k_3$ is the product of the primes in $(P_{m^{(1)}_0} \cap P_{m^{(2)}_1})$ and $k_4$ is the product of the primes in $(P_{m^{(2)}_0} \cap P_{m^{(1)}_1})$.

Let us argue that $k_1,\dots,k_4$ are (pairwise) coprime.
First, since $m^{(2)}_0m^{(2)}_1$ divides $n^{(2)}$, it follows that $\gcd(k_1,k_2)=\gcd(k_2,k_3)=\gcd(k_1,k_4)=1$.
Likewise, since $m^{(1)}_0m^{(1)}_1$ divides $n^{(1)}$, it follows that $\gcd(k_2,k_3)=\gcd(k_2,k_4)=1$.
Finally, since $m^{(1)}_0,m^{(1)}_1$ are coprime, it follows that $\gcd(k_3,k_4)=1$.
Hence, $k_1,\dots,k_4$ are coprime.
Consequently,
\[ \Q_{k_1^\infty}/\Z \oplus \cdots \oplus \Q_{k_4^\infty}/\Z \cong \Q_{(k_1\cdots k_4)^\infty}/\Z = \Q_{m_0}/\Z, \]
and therefore,
\[ K_0(A) \cong \Q_n \oplus \Q_{m_0}/\Z. \]

Essentially the same argument shows that $K_1(A) \cong \Q_{m_1}/\Z$.
Also 
\[ [1_A]_0 = [1_{\mathcal E_{n^{(1)},m^{(1)}_0,m^{(1)}_1}}]_0 \otimes [\mathcal E_{n^{(1)},m^{(1)}_0,m^{(1)}_1}]_0 = 1 \oplus 0 \in \Q_n \oplus \Q_{m_0}/\Z. \]
Hence, by classification (see Remark \ref{rmk:EFdef}), $A \cong \mathcal E_{n,m_0,m_1}$.
\end{proof}

Essentially the same arguments can be used to derive the next two computations.

\begin{prop}
Let $n,m^{(1)}_0,m^{(1)}_1,m^{(2)}_0,m^{(2)}_1$ be supernatural numbers of infinite type, such that $m^{(i)}_0,m^{(i)}_1$ are coprime for $i=1,2$ and $m^{(1)}_0m^{(1)}_1$ divides $n$.
Then $\mathcal E_{n,m^{(1)}_0,m^{(1)}_1} \otimes \mathcal F_{m^{(2)}_0,m^{(2)}_1} \cong \mathcal F_{m_0,m_1}$ where:
\begin{enumerate}
\item $m_0$ is the product of the following set of primes, each taken infinitely many times
\[ (P_{m^{(2)}_0} \setminus P_{n}) \cup (P_{m^{(1)}_0} \cap P_{m^{(2)}_1}) \cup (P_{m^{(2)}_0} \cap P_{m^{(1)}_1}); \quad \text{and} \]
\item $m_1$ is the product of the following set of primes, each taken infinitely many times
\[ (P_{m^{(2)}_1} \setminus P_{n}) \cup (P_{m^{(1)}_1} \cap P_{m^{(2)}_1}) \cup (P_{m^{(2)}_0} \cap P_{m^{(1)}_0}). \]
\end{enumerate}
\end{prop}

\begin{prop}
Let $m^{(1)}_0,m^{(1)}_1,m^{(2)}_0,m^{(2)}_1$ be supernatural numbers of infinite type, such that $m^{(i)}_0,m^{(i)}_1$ are coprime for $i=1,2$.
Then $\mathcal F_{m^{(1)}_0,m^{(1)}_1} \otimes \mathcal F_{m^{(2)}_0,m^{(2)}_1} \cong \mathcal F_{m_0,m_1}$ where:
\begin{enumerate}
\item $m_0$ is the product of the following set of primes, each taken infinitely many times
\[ (P_{m^{(1)}_0} \cap P_{m^{(2)}_1}) \cup (P_{m^{(2)}_0} \cap P_{m^{(1)}_1}); \quad \text{and} \]
\item $m_1$ is the product of the following set of primes, each taken infinitely many times
\[ (P_{m^{(1)}_1} \cap P_{m^{(2)}_1}) \cup (P_{m^{(2)}_0} \cap P_{m^{(1)}_0}). \]
\end{enumerate}
\end{prop}

Let $A$ be a classifiable C*-algebra with approximately inner flip (recall from Remark \ref{rmk:MainThmClass} that these are precisely C*-algebras in Theorem \ref{thm:MainThm} with approximately inner flip).
Then, by considering the various cases from Theorem \ref{thm:MainThm} (v), and using the above computations as appropriate, $A \otimes A$ is stably isomorphic to one of
\begin{enumerate}
\item $\mathcal Z$;
\item $\mathcal E_{n,1,m}$ where $m,n$ are supernatural numbers of infinite type and $m$ divides $n$;
\item $\mathcal E_{n,1,m} \otimes \mathcal O_\infty$ where $m,n$ are supernatural numbers of infinite type and $m$ divides $n$ (this occurs in both cases (b) or (c) of Theorem \ref{thm:MainThm} (v)); or
\item $\mathcal F_{1,m}$.
\end{enumerate}
Each of these C*-algebras is self-absorbing (isomorphic to its tensor product with itself), which establishes the corollary below.
Among them, $\mathcal Z$, $\mathcal E_{n,1,1}$, $\mathcal E_{n,1,1} \otimes \mathcal O_\infty$, and $\mathcal F_{1,1}$ are strongly self-absorbing (recall that $\mathcal E_{n,1,1} \cong M_n$ and $\mathcal F_{1,1} \cong \mathcal O_2$), hence each of the others is not isomorphic to their own infinite tensor product (by \cite[Proposition 1.9]{TomsWinter:ssa}).
In fact, without much effort, one can show using classification that $\mathcal E_{n,1,m}^{\otimes \infty} \cong M_n$ and $\mathcal F_{1,m}^{\otimes \infty} \cong \mathcal O_2$.

\begin{cor}
\label{cor:AAsa}
Let $A$ be a classifiable C*-algebra with approximately inner flip.
Then $A \otimes A \otimes$ is stably self-absorbing (i.e., $A^{\otimes 2} \otimes \mathcal K \cong A^{\otimes 4} \otimes \mathcal K$).
\end{cor}

We may organize the self-absorbing classifiable C*-algebras with approximately inner flip (listed above) by ordering them according to absorption ($A \prec B$ if $B \cong A \otimes B$).
I would like to thank Mikael R\o rdam for suggesting to do so.
The following diagram captures this ordering (lower algebras absorb the higher algebras).
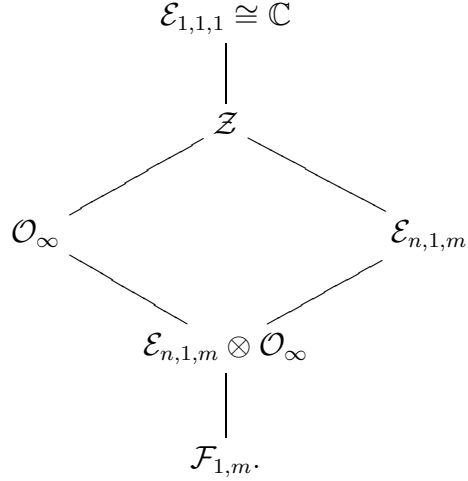
\[ \xymatrix{
& \mathcal E_{1,1,1} \cong \mathbb C \ar@{-}[d] & \\
& \mathcal Z \ar@{-}[dr] \ar@{-}[dl] & \\
\mathcal O_\infty \ar@{-}[dr] && \mathcal E_{n,1,m} \ar@{-}[dl] & \\
&\mathcal E_{n,1,m} \otimes \mathcal O_\infty \ar@{-}[d] & \\
&\mathcal F_{1,m}. &
} \]
Within the family of algebras of the form $\mathcal E_{n,1,m}$, we have $\mathcal E_{n,1,m} \cong \mathcal E_{n,1,m} \otimes \mathcal E_{n',1,m'}$ iff $n'|n$ and $m|m'$.
Hence, locally the ordering on this family looks like the following, where $m,n,p,q$ are supernatural numbers of infinite type such that $m,p$ are coprime, $n,q$ are coprime, and $m$ and $p$ both divide $n$.
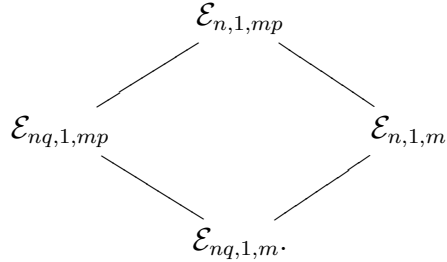
\[ \xymatrix{
& \mathcal E_{n,1,mp} \ar@{-}[dr] \ar@{-}[dl] & \\
\mathcal E_{nq,1,mp} \ar@{-}[dr] && \mathcal E_{n,1,m} \ar@{-}[dl] & \\
& \mathcal E_{nq,1,m}. &
} \]
The ordering within the families of algebras of the form $\mathcal E_{n,1,m} \otimes \mathcal O_\infty$ and $\mathcal F_{1,m}$ are similar.

\newcommand{\cstar}{C*}

\end{document}